\documentclass[10pt,onecolumn]{IEEEtran}

\usepackage{amssymb,amsmath, amsfonts, theorem}
\usepackage{graphicx}
\usepackage{epsfig}
\usepackage{algorithm,algorithmic,xspace}
\usepackage{times}
\usepackage{color}
\usepackage{flafter}
\usepackage{psfrag}
\usepackage{pstricks}
\usepackage[tight,footnotesize]{subfigure}
\usepackage{bbm}
\usepackage{float}
\usepackage{pgfplots}
\pgfplotsset{compat=1.3}
\usepackage[Symbol]{upgreek}
\usepackage{tikz}
\usetikzlibrary{matrix}
\usetikzlibrary{positioning}
\usetikzlibrary{arrows,shapes}
\usepackage{pst-pdf}
\usepackage{subfigure}
\usepackage{cases}
\usepackage{enumerate}
\usepackage{authblk}
\usepackage{fullpage}
\usepackage{cite}
\usepackage{epstopdf}

 \newcommand{\mc}{\mathcal}

\newcommand{\reals}{{\mathbb R}}

\newcommand{\ones}{\mathbbm{1}}
\newcommand{\ms}{\scriptscriptstyle}
\newcommand{\Etree}{E_{{\ms \mc{T}}}}
\newcommand{\Eforrest}{E_{{\ms \mc{F}}}}
\newcommand{\Ecycle}{E_{{\ms \mc{C}}}}
\newcommand{\Etreep}{E_{{\ms \mc{T}_+}}}
\newcommand{\Eforrestp}{E_{{\ms \mc{F}_+}}}

\newcommand{\Eforrestm}{E_{{\ms \mc{F}_-}}}

\newcommand{\Tcycle}{T_{{\ms (\mc{T},\mc{C})}}}
\newcommand{\Tcyclef}{T_{{\ms (\mc{F},\mc{C})}}}

\newcommand{\Tplus}{T_{{\ms (\mc{F}_+,\mc{C}_+)}}}
\newcommand{\Tminus}{T_{{\ms (\mc{F}_{-},\mc{C}_{-})}}}
\newcommand{\R}{R_{\ms (\mc{T},\mc{C})}}
\newcommand{\Rf}{R_{\ms (\mc{F},\mc{C})}}
\newcommand{\Rplus}{R_{\ms (\mc{F}_{+},\mc{C}_{+})}}
\newcommand{\Rtplus}{R_{\ms (\mc{T}_{+},\mc{C}_{+})}}
\newcommand{\Rminus}{R_{\ms (\mc{F}_{-},\mc{C}_{-})}}

\newcommand{\Edgelapf}{L_e(\mc{F})R_{\ms (\mc{F},\mc{C})}WR_{\ms (\mc{F},\mc{C})}^T}

\newcommand{\w}{\mathrm v}

 \newtheorem{theorem}{Theorem}[section]
 \newtheorem{proposition}[theorem]{Proposition}
 \newtheorem{corollary}[theorem]{Corollary}
 \newtheorem{definition}[theorem]{Definition}
 \newtheorem{lemma}[theorem]{Lemma}

 \def\QEDclosed{\mbox{\rule[0pt]{1.3ex}{1.3ex}}} 

\def\QED{\QEDopen}

\newcommand{\bea}{\begin{eqnarray}}
\newcommand{\eea}{\end{eqnarray}}
\newcommand{\beas}{\begin{eqnarray*}}
\newcommand{\eeas}{\end{eqnarray*}}
\newcommand{\leftm}{\left[\begin{array}}
\newcommand{\rightm}{\end{array}\right]}

 \def\QEDclosed{\mbox{\rule[0pt]{1.3ex}{1.3ex}}} 
 
 \def\QED{\QEDclosed} 
\newcommand{\margin}[1]{\marginpar{\tiny\color{red} #1}}

\newcommand{\MBremove}[1]{\margin{removed by MB}}
\newcommand\oprocendsymbol{\hbox{$\square$}}
\newcommand\oprocend{\relax\ifmmode\else\unskip\hfill\fi\oprocendsymbol}

\title{\LARGE \bf On the Robustness of Uncertain Consensus Networks}

\author{Daniel Zelazo$^1$ and Mathias B\"urger$^2$
\thanks{$^1$Daniel Zelazo is with the Department of Aerospace Engineering, Israel Institute of Technology, Israel.  
        {\tt\small dzelazo@technion.ac.il}.}
\thanks{$^2$Mathias B\"urger is with the Cognitive Systems Group at Robert Bosch GmbH
        {\tt\small mathias.buerger@de.bosch.com}.}%
}

\begin{document}
\maketitle
\thispagestyle{empty}
\pagestyle{empty}

\begin{abstract}
This work considers the robustness of uncertain consensus networks.  The first set of results studies the stability properties of consensus networks with negative edge weights.  We show that if either the negative weight edges form a cut in the graph, or any single negative edge weight has magnitude less than the inverse of the effective resistance between the two incident nodes, then the resulting network is unstable.  These results are then applied to analyze the robustness properties of the consensus network with additive but bounded perturbations of the edge weights.  It is shown that the small-gain condition is related again to cuts in the graph and effective resistance.  For the single edge case, the small-gain condition is also shown to be exact.  The results are then extended to consensus networks with non-linear couplings.
\end{abstract}

\section{Introduction}

The consensus protocol has recently emerged as a canonical model for the study of \emph{networked dynamic systems}.  In the linear setting, the consensus protocol, comprised of a collection of single integrator dynamic agents interacting over an information exchange network (the graph), has been studied from both dynamic systems and graph theoretic perspectives~\cite{Mesbahi2010}.  The most basic setting considered in consensus networks assumes an undirected connected graph with non-negative weights on the edges of the graph. In such a setting, it is well known that the trajectories of all agents in the network converge to a common value.  The use of edge weights in these networks often arise from the modeling of physical processes \cite{Newman2003}, or as a design parameter used to improve certain performance metrics \cite{Shafi2010, Xiao2004}.

Recently, there has been a growing interest in multi-agent networks containing negative edge weights.  For example, in the work \cite{Xiao2004} it was shown that negative edge weights can appear as an optimal solution for finding the fastest converging linear iteration used in distributed averaging.   Negative edge weights in problems related to the control of multi-agent systems can also lead to steady-state configurations that are \emph{clustering} \cite{Qin2013, Xia2011a}.  In \cite{Altafini2013b}, negative weights are used to model antagonistic interactions in a social network and conditions are provided for when such weights lead to \emph{bipartite consensus}.  The work of \cite{Bronski2014} provides bounds on the number of positive, negative, and zero eigenvalues of the weighted Laplacian matrix with negative weights.

An important issue that has not received much attention in the controls community is the \emph{robustness} of these linear weighted consensus protocols.  Indeed, even in the most basic setups, the trajectories generated by a weighted consensus protocol can be very rich including steady-state trajectories that are synchronized, clustering, or unstable.   If the weights in a consensus network arise from its physical modeling, then one must consider the performance of such systems when the weights are not known exactly.  Similarly, if the weights are designed as an engineered parameter, then it is important to consider the robustness of the system in the presence of malicious attacks on the network, one of which could be the manipulation of the nominally designed edge weights.

The question of robustness in this context relates to the underlying graph.  Recent works have analyzed robustness of synchronization networks with random link failures \cite{Diwadkar2011}, particular families of graphs in the context of the $\mc{H}_2$ performance \cite{Young2011}, and uncertain tree structures \cite{Sandhu2005}. In \cite{Trentelman2013}, uncertainties in the dynamics of each agent comprising a networked system was studied in the context of the small-gain theorem.  However, these works do not address the fundamental \emph{combinatorial} aspect that can arise in uncertain networked systems.

This then motivates the main contributions of this work.  We aim to study the robustness properties of linear weighted consensus protocols and provide graph-theoretic interpretations of the results.  We consider consensus networks where the nominal edge weights are subject to some bounded additive perturbations.  At the heart of this analysis is an algebraic and graph-theoretic characterization of the definiteness of the weighted graph Laplacian with negative edge weights.  We then apply these results to a linear consensus protocol in the context of the celebrated \emph{small gain theorem}.  We show that in certain cases, the small gain theorem provides an exact robustness margin for the uncertain consensus network.  Furthermore, this measure turns out to be related to the notion of the \emph{effective resistance} of a graph.  These results are also extended to consider edge weights with nonlinear, but sector bounded, perturbations.  

The organization of this paper is as follows.  In Section \ref{sec:edgelap}, fundamental notions from graph theory are reviewed, including the results on the weighted graph and edge Laplacian matrices \cite{Zelazo2009b, Zelazo2011, Zelazo2014}.  Section \ref{sec:consensus} introduces the uncertain consensus models.  Section \ref{sec:negweights_consensus} provides a general analysis on the stability of linear consensus protocols with negative edge weights, and Section \ref{sec:small gain} leverages these results to make statements on the robust stability of uncertain networks.  Finally, concluding remarks are offered in Sections \ref{sec:conclusion}.

\paragraph*{Preliminaries}
The notations employed in this work are standard.  The $n$-dimensional Euclidean vector space is denoted $\reals^n$ and $\|x\|$ denotes the standard Euclidean norm for $x \in \reals^n$.  For a real $n \times m$ matrix $A$, $\overline{\sigma}(A)$ denotes its largest singular value and the spectral norm of the matrix $A$ is $\|A\| = \overline{\sigma}(A)$.  The image and null-space of a matrix $A$ are denoted $\mbox{\bf IM}[A]$ and $\mc{N}[A]$, respectively.The signature of a real symmetric matrix $A$, denoted as the triple $s(A) = (n_+,n_-,n_0)$, is the number of positive, negative, and zero eigenvalues of the matrix.  An important result on the the signature of a matrix is \emph{Sylvester's Law of Inertia}, which states that all congruent symmetric matrices have the same signature \cite{Horn1985}.\footnote{A square matrix $A$ is \emph{congruent} to a square matrix $B$ of the same dimension if there exists an invertible matrix $S$ such that $B=S^TAS$.}  The $m$-dimensional space of piecewise continuous vector-valued square integrable functions is denoted as $\mc{L}^m_2[0, \infty)$. 
 For a linear dynamical system described by the state-space model $\dot{x}(t)=Ax(t)+Bu(t), \, y(t)=Cx(t)+Du(t)$, its transfer function is denoted as $G(s)=C(sI-A)^{-1}B+D$.  The $\mc{L}_2$-induced gain of the system described by $G(s)$ is determined by the $\mc{H}_{\infty}$ system norm, $\|G(s)\|_{\infty} = \sup_\omega \overline{\sigma}(G(j\omega))$ \cite{Dullerud2000}.

\section{Forests, Cycles, and the Edge Laplacian}\label{sec:edgelap}

In this section we introduce the fundamental notions of spanning forests and cycles in graphs and present their associated matrix representations \cite{Godsil2001}. We consider undirected weighted graphs described by the triple $\mc{G} = (\mc{V},\mc{E},\mc{W})$ consisting of the node set $\mc{V}$, edge set $\mc{E} \subseteq \mc{V} \times \mc{V}$, and weight function that maps each edge to a scalar value, $\mc{W}: \mc{E} \rightarrow \reals $.
 We often collect the weights of all the edges in a diagonal matrix $W \in \reals^{|\mc{E} | \times |\mc{E} |}$ such that $[W]_{kk} = \mc{W}(k)=w_k$ where $e_k = (i,j) \in \mc{E}$.  
 
The maximal acyclic subgraph of $\mc{G}$ is referred to as a \emph{spanning forrest}, denoted $\mc{F} = (\mc{V},\mc{E}_{\ms \mc{F}}) $.  If $\mc{F}$ is connected, it is called a \emph{spanning tree}, and is denoted $\mc{T} = (\mc{V},\mc{E}_{\ms \mc{T}})$.
Every graph $\mc{G}$ can always be expressed as the union of a spanning tree (spanning forrest) and another subgraph containing the remaining edges, i.e., $\mc{G} = \mc{T} \cup \mc{C}$ ($\mc{G} = \mc{F} \cup \mc{C}$).  The subgraph $\mc{C}$ necessarily ``completes cycles" in $\mc{G}$, and is defined as $\mc{C} = (\mc{V},\mc{E}_{\ms \mc{C}}) \subset \mc{G}$ with $\mc{E}_{\ms \mc{C}} = \mc{E} \setminus \mc{E}_{\ms \mc{T}}$ (similarly defined with a forrest); we refer to this as the \emph{cycle subgraph}.  

The \emph{incidence matrix} of a graph (with arbitrary edge orientation), $E(\mc{G}) \in \reals^{|\mc{V}| \times |\mc{E}|}$, is such that for edge $e_k = (i,j) \in \mc{E}$, $[E(\mc{G})]_{i,k} = 1$, $[E(\mc{G})]_{j,k} = -1$, and $[E(\mc{G})]_{\ell,k} = 0$ for $\ell \neq i,j$.  With an appropriate labeling of the edges, we can always express the incidence matrix as { $E(\mc{G}) = \leftm{cc} E(\mc{T}) & E(\mc{C}) \rightm$} ({$E(\mc{G}) = \leftm{cc} E(\mc{F}) & E(\mc{C}) \rightm$}).  An important property of the incidence matrix is that $E(\mc{G})^T\ones = 0$ for any graph $\mc{G}$, where $\ones$ is the vector of all ones, and that $\mbox{\bf rk}[E(\mc{G})]  = n-c$ for a graph with $c$ connected components.  For a more compact notation, we will write $E := E(\mc{G}), \Etree = E(\mc{T}), \Eforrest = E(\mc{F})$, and $\Ecycle = E(\mc{C})$.  

The weighted graph Laplacian of $\mc{G}$ is defined as $L(\mc{G}) = EWE^{T} \in \reals^{|\mc{V}| \times |\mc{V}|}$.  It is well known that for a weighted graph with only positive edge weights (i.e., $\mc{W}:\mc{E} \rightarrow \reals_{\geq 0}$), the signature of the Laplacian can be expressed as $s(L(\mc{G})) = (|\mc{V}|-c,0,c)$, where $c$ is the number of connected components of $\mc{G}$~\cite{Godsil2001}.   Another symmetric matrix representation of a graph is the \emph{edge Laplacian} \cite{Zelazo2009b}.  For weighted graphs, we define the edge Laplacian as $L_e(\mc{G}) := W^{\frac{1}{2}}E^TEW^{\frac{1}{2}} \in \reals^{|\mc{E}| \times |\mc{E}|}$.

\begin{table*}[t]
\begin{center}
\begin{tabular}{| l | l || l | l |}
\hline
 Incidence Matrix & $E(\mc{G}) = \leftm{cc} E(\mc{F}) & E(\mc{C}) \rightm $ & Graph Laplacian & $ L(\mc{G})=EWE^T$ \\ \hline
 Tucker Representation & $ \Tcyclef = L_e(\mc{F})^{-1}\Eforrest^T\Ecycle$ & Edge Laplacian & $ L_e(\mc{G})=W^{\frac{1}{2}}E^TEW^{\frac{1}{2}} $\\ \hline
 Cut-set Matrix & $\Rf = \leftm{cc} I & \Tcyclef \rightm$ & Essential Edge Laplacian & $L_{ess}(\mc{F}) = L_e(\mc{F})\Rf W \Rf^T$ \\ \hline
\end{tabular}
\end{center}
\caption{Summary of matrix representations for graphs.}\label{table.matrixgraphs}
\vspace{-20pt}
\end{table*}

We now review some basic results relating the weighted edge Laplacian matrix to the graph Laplacian.  
\begin{proposition}\label{prop:edgelap_sim2}
The weighted Laplacian matrix $L(\mc{G})=EWE^T$ is similar to the matrix
\bea\label{essential_weighted_edgelap}
\leftm{cc} L_e(\mc{F}) \Rf W\Rf^T & {\bf 0} \\ {\bf 0} & {\bf 0}_c \rightm,
\eea
where $\mc{G}$ has $c$ connected components, $\mc{F} \subseteq \mc{G}$ is a spanning forrest 
and $\Rf = \leftm{cc} I & L_e(\mc{F})^{-1}\Eforrest^T\Ecycle \rightm=  \leftm{cc} I & \Tcyclef \rightm.$
\end{proposition}
\begin{proof}
For a spanning forest, one has $\mbox{\bf rk}[E_{\ms \mc{F}}]  = n-c$ and therefore $L_e(\mc{F})=E_{\ms \mc{F}}^TE_{\ms \mc{F}}$ is invertible.  Define the transformation matrices 
\beas
V = \leftm{cc} \Eforrest& N_{\ms {\mc{F}}} \rightm &,&
V^{-1} = \leftm{c} L_e(\mc{F})^{-1}\Eforrest^T \\ N_{\ms {\mc{F}}}^T \rightm,
\eeas
where $\mbox{\bf IM} [N_{\ms \mc{F}} ]= \mbox{\bf span}[\mc{N}[\Eforrest^T]]$.
It is straightforward to verify that the matrix in (\ref{essential_weighted_edgelap}) equals $V^{-1}L(\mc{G})V$.
\hfill \QEDclosed
\end{proof}

The matrix $L_e(\mc{F}) \Rf W\Rf^T:=L_{\ms ess}(\mc{F})$ is referred to as the \emph{essential edge Laplacian} \cite{Zelazo2011} (for a connected graph with spanning tree $\mc{T}$, we write $L_{\ms ess}(\mc{T})$).  Note that the essential edge Laplacian is therefore a non-singular matrix, and the edge Laplacian of a spanning forest ($L_e(\mc{F})$) is always positive-definite.
The matrix $L_e(\mc{F})^{-1}\Eforrest^T\Ecycle:=\Tcyclef$ is sometimes referred to as the \emph{Tucker representation} of a graph \cite{Rockafellar1997}.  The rows of the matrix $\Rf$ form a basis for the \emph{cut space} of a graph \cite{Godsil2001} and we refer to it as the \emph{cut-set matrix}.
For a more in depth discussion on the matrices $\Rf$ and $\Tcyclef$, please see \cite{Zelazo2009b, Zelazo2011}.  Note also that the matrix $L_e(\mc{F})^{-1}\Eforrest^T$ is the \emph{left-inverse} of $\Eforrest$; we denote this matrix as $\Eforrest^L$.  For the readers convenience, we summarize the various matrix representations of graphs in Table \ref{table.matrixgraphs}.


\begin{theorem}[\cite{Zelazo2014}]\label{thm:signature1}
Assume $\mc{G}$ has $c$ connected components and $s(L(\mc{G}))=(n_+,n_-,n_0)$.  Then $s(L_{\ms ess}(\mc{F})) = (n_+,n_-,n_0-c)$. Furthermore, $L_{\ms ess}(\mc{F})$ has the same signature as the matrix $\Rf W \Rf^T$. 
\end{theorem} 

\begin{proof}
Using the similarity transformation matrix $L_e(\mc{F})^{\frac{1}{2}}$ we have that $\Edgelapf$ is similar to $L_e(\mc{F})^{\frac{1}{2}}\Rf W \Rf^TL_e(\mc{F})^{\frac{1}{2}}$.  This matrix is congruent to $\Rf W \Rf^T$ and thus by Sylvester's Law of Inertia has the same inertia as $\Edgelapf$.   \hfill \QEDclosed
\end{proof}

An immediate corollary of Theorem \ref{thm:signature1} is that the if $s(L(\mc{G}))=(n_+,n_-,n_0)$ and $\mc{G}$ has $c$ connected components, then $s(\Rf W \Rf^T) = (n_+,n_-,n_0-c)$.

The matrix $\Rf W\Rf^T$ turns out to be closely related to many combinatorial properties of a graph.  For example, the rows of the matrix $\Rf$ form a basis for the cut-space of the graph \cite{Godsil2001}.  This matrix is also intimately related to the notion of effective resistance of a graph, which will be discussed in Section \ref{sec:negweights_consensus}.  Theorem \ref{thm:signature1} thus shows that studying the definiteness of the weighted Laplacian can be reduced to studying the matrix $\Rf W \Rf^T$ which contains in a more explicit way information on how both the location and magnitude of negative weight edges influence it spectral properties.  This theme will be periodically revisited throughout this work.

\section{Consensus and the Edge Agreement }\label{sec:consensus}
In this work we consider the linear weighted consensus protocol \cite{Mesbahi2010} in the presence of exogenous finite-energy disturbances.  The consensus dynamics over the graph $\mc{G}=(\mc{V},\mc{E},\mc{W})$ can thus be expressed as
\bea\label{consensus}
\Sigma(\mc{G}) \,:\, \left\{ 
	\begin{array}{ccc}  
		\dot{x}(t) &=& -L(\mc{G})x(t) + \w(t) \\
		z(t) &=& E(\mc{G}_o)^Tx(t)	
	\end{array} 
	\right. .
\eea
Here, $\w(t) = \leftm{ccc} \w_1(t) & \cdots & \w_n(t) \rightm^T \in \mc{L}_2^n[0,\infty)$ is a finite-energy exogenous disturbance entering each agent in the system.  The performance of the system is measured in terms of the energy of the vector $z(t) \in \reals^{|\mc{E}_{o}|}$ capturing the relative states over a set of edges determined by the unweighted graph $\mc{G}_o=(\mc{V}, \mc{E}_{o})$.\footnote{Note that $\mc{G}_o$ need not have any dependence on $\mc{G}$.}  

With this model we can consider how finite energy disturbances affect the asymptotic deviation a subset of the states to an agreement value.  
This can formally be analyzed by considering the $\mc{L}_2$-induced gain of the system, i.e., its $\mc{H}_{\infty}$ performance, $\|\Sigma(\mc{G})\|_{\infty}$ \cite{Zelazo2009b}.   It is a well-established result that if the graph $\mc{G}$ is connected and all the weights are positive, then (\ref{consensus}) 
reaches consensus, i.e., $\lim_{t \rightarrow \infty} z(t) = 0$, for any output graph $\mc{G}_o$.  Therefore, $z(t) \in \mc{L}_2^{|\mc{E}_{o}|}[0,\infty)$ and $\|\Sigma(\mc{G})\|_{\infty}$ is finite.  However, this is not the case if $L(\mc{G})$ has multiple eigenvalues at the origin (and certainly if any eigenvalue of $L(\mc{G})$ is positive).  Indeed, if $\mc{G}$ is disconnected and $E(\mc{G}_o) = \overline{\mc{G}}$,\footnote{The notation $\overline{\mc{G}}$ denotes the \emph{complement} of the graph $\mc{G}$.} then $z(t)$ is not a finite-energy signal.  On the other hand, for the same example and choosing $\mc{G}_o \subseteq \mc{G}$ it can be verified that $z(t) \in \mc{L}_2^{|\mc{E}_o|}[0,\infty)$ (since each component of the graph will reach consensus).  

Nevertheless, we would like to examine some notion of performance of (\ref{consensus}) for any choice of output graph $\mc{G}_o$.  In this direction, we can define a coordinate transformation for (\ref{consensus}), $\tilde{x}(t) = S^{-1}x(t)$, with 
$$S = \leftm{cc} (\Eforrest^L)^T & N_{\ms \mc{F}}\rightm, \; S^{-1} = \leftm{c}  \Eforrest^T \\ N_{\ms \mc{F}}^T \rightm,$$
where ${\bf IM}[N_{\ms \mc{F}}] = \mbox{\bf span}\{\mc{N}[\Eforrest^T]\}$.  Thus, the state vector $\tilde{x}(t)$ can be partitioned into two components as $\tilde{x}(t) = \leftm{cc} x_{\ms \mc{F}}(t)^T & x_{a}^T(t) \rightm^T$, where $x_{\ms \mc{F}}(t) = \Eforrest^T x(t)$ are the relative states over the edges forming the spanning forrest of $\mc{G}$, and $x_a(t) = N_{\ms \mc{F}}^Tx(t)$ correspond to modes in the direction of the all-ones vector across each component of $\mc{G}$.  Applying this transformation to (\ref{consensus}) leads to the following system,
{
\bea\label{consensus_transform}
\leftm{c} \dot{x}_{\ms \mc{F}}(t) \\ \dot{x}_a(t) \rightm &=&\leftm{cc} -L_{\ms ess}(\mc{F})& {\bf 0} \\ {\bf 0} & {\bf 0} \rightm \leftm{c} x_{\ms \mc{F}}(t) \\ x_a(t) \rightm +  S^{-1} \w(t) \nonumber \\
z(t) &=&E(\mc{G}_o)^T\leftm{cc} (\Eforrest^L)^T&N_{\ms \mc{F}}\rightm  \leftm{c} x_{\ms \mc{F}}(t) \\ x_a(t) \rightm	.
\eea
}
In the new coordinate system, it is now straightforward to show that $\lim_{t \rightarrow \infty} z(t) = E(\mc{G}_o)^TN_{\ms \mc{F}} x_{a}(0)$ and furthermore that $E(\mc{G}_o)^T(\Eforrest^L)^T x_{\ms \mc{F}}(t) \in \mc{L}_2^{|\mc{E}_{{\ms \mc{F}}}|}[0,\infty)$.  We can now consider the following truncated system,
\bea\label{consensus_F}
\Sigma_{\ms \mc{F}}(\mc{G}) : \left\{ 
	\begin{array}{ccc}  
		\dot{x}_{\ms \mc{F}}(t) &=&-L_{\ms ess}(\mc{F})x_{\ms \mc{F}}(t) + \Eforrest^T\w(t) \\
		z_{\ms \mc{F}}(t) &=& E(\mc{G}_o)^T(\Eforrest^L)^Tx_{\ms \mc{F}}(t)	
	\end{array} 
	\right. .
\eea
We term the system $\Sigma_{\ms \mc{F}}(\mc{G})$ the \emph{edge agreement protocol} over the spanning forest $\mc{F} \subseteq \mc{G}$.  The above transformation holds also when $\mc{G}$ is connected and $\mc{F} = \mc{T}$ is a spanning tree.  
It is verifiable that the system $\Sigma_{\ms \mc{F}}(\mc{G})$ is \emph{minimal}. 
 When $\mc{F} = \mc{T}$ is a spanning tree, then $\Sigma_{\ms \mc{F}}(\mc{G})$ is a minimal realization of $\Sigma(\mc{G})$ \cite{Zelazo2011}.

\subsection{The Uncertain Edge Agreement Protocol}\label{subsec:uncertain_edge}
\begin{figure}[!t]
\begin{center}
\includegraphics[scale=0.6]{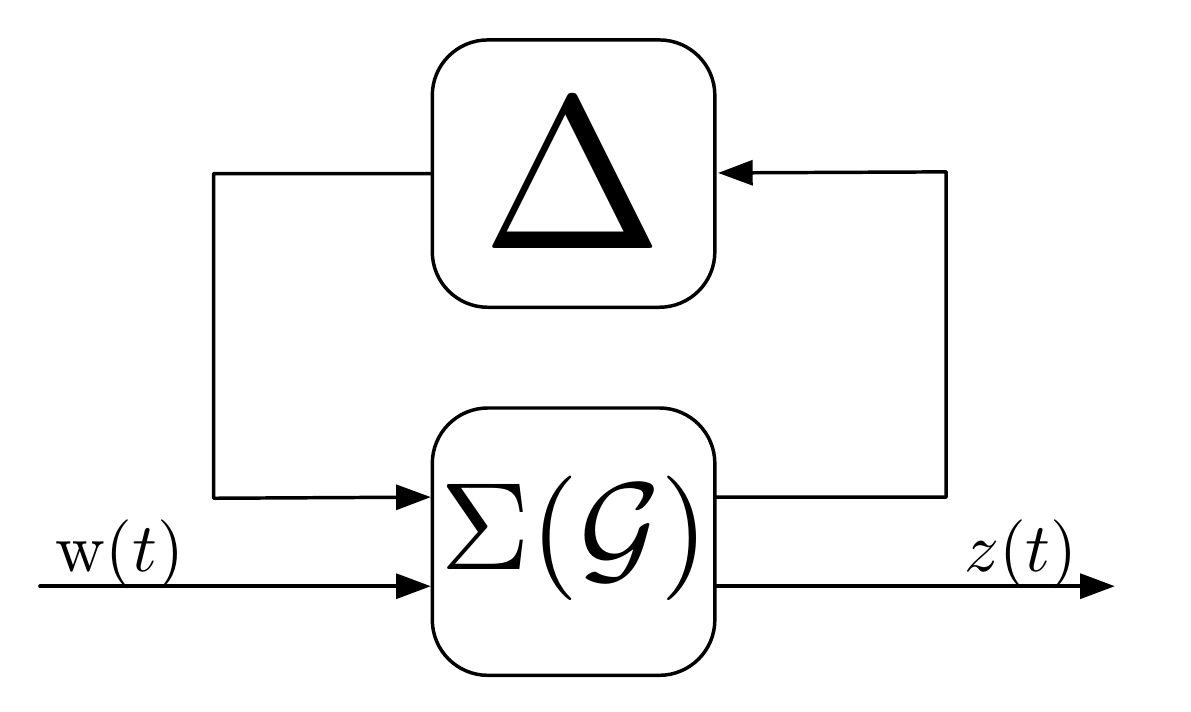}
\caption{The consensus network with uncertainties.}\label{fig:uncertain}
\end{center}
\end{figure}

We now introduce a notion of uncertainty into the edge agreement protocols.  
First we examine an uncertainty model where it is assumed that the exact weights of a subset of edges are an uncertain but bounded perturbation about some nominal value.  In this direction, let $\mc{E}_{\ms \Delta} \subseteq \mc{E}$ denote the set of uncertain edges.  The nominal edge weight for an edge $k \in \mc{E}_{\Delta}$ is determined by the weight function $\mc{W}$; i.e., the nominal weight of edge $k$ is $\mc{W}(k) = w_k$.  The uncertainty of the weight on edge $k$ is modeled as an additive perturbation to the nominal edge weight as $w_k + \delta_k$ with $|\delta_k| \leq \overline{\delta}$ for some finite positive scalar $\overline{\delta}$.  Thus, we can define the uncertainty set as 
\beas
{\bf \Delta} &\hspace{-5pt}=&\hspace{-5pt} \{\Delta \, : \, \Delta = \mbox{\bf diag}\{\delta_1,\ldots,\delta_{|\mc{E}_{\ms \Delta}|}\}, \,  \|\Delta\| \leq \overline \delta \}.
\eeas

In this way, we can consider the uncertain edge agreement protocol as 
{
\bea\label{uncertain_edge_agreement}
\Sigma_{\ms \mc{F}}(\mc{G},\Delta): &&\left\{\begin{array}{ccc} \dot{x}_{\ms \mc{F}}(t) &=& -L_e(\mc{F})R_{\ms (\mc{F},\mc{C})}\left(W+P\Delta P^T\right)R_{\ms (\mc{F},\mc{C})}^T x_{\ms \mc{F}}(t) + \Eforrest^T {\mathrm v}(t)\\
z(t) &=&E(\mc{G}_{o})^T (\Eforrest^L)^T x_{\ms \mc{F}}(t) \end{array} \right. . \nonumber 
\eea
}
for $\Delta \in {\bf \Delta}$.
In this form we see that the uncertainty is a structured additive uncertainty.  
The matrix $P \in \reals^{|\mc{E}| \times |\mc{E}_{\ms \Delta}|}$ is a $\{0,1\}$-matrix used to select the uncertain edges with $[P]_{ij} = 1$ if $e_i \in \mc{E} \cap \mc{E}_{\ms \Delta}$, and $[P]_{ij} = 0$ otherwise (i.e., $E(\mc{G})P = E(\mc{G}_{\ms \Delta})$ with $\mc{G}_{\ms \Delta}=(\mc{V},\mc{E}_{\ms \Delta})$).  This setup is visualized by the two-port block diagram in Figure \ref{fig:uncertain}.

We also consider the consensus protocol with non-linear couplings.  By introducing appropriate assumptions on the non-linear couplings we are able to cast the problem as an uncertain agreement protocol in the form of Figure \ref{fig:uncertain}.  In this direction, the non-linear consensus protocol has the form
\beas
\dot{x}(t) &=& -L(\mc{G})x(t) - E(\mc{G}_{\ms \Delta}) \Phi \left(E(\mc{G}_{\ms \Delta})^T x(t)\right) + {\mathrm v}(t) \\
z(t) &=& E(\mc{G}_o)^Tx(t).
\eeas
As before, we assume $\mc{E}_{\ms \Delta}\subseteq \mc{E}$.
 The non-linear vector function $\Phi : \reals^{| \mc{E}_{\ms \Delta}|} \rightarrow \reals^{|\mc{E}_{\ms \Delta}|}$ is assumed to be decoupled, that is $\Phi(y) = \leftm{ccc} \phi_1(y_1) & \cdots & \phi_{| \mc{E}_{\ms \Delta}|}(y_{| \mc{E}_{\ms \Delta}|} )\rightm^T$.  Furthermore, we assume that the nonlinear functions $\phi_i(\cdot)$ belong to the sector $[\alpha_i, \, \beta_i]$; that is $\alpha_i u_i^2 \leq u_i \phi_i(y_i) \leq \beta_i u_i^2$ for all $u_i \in \reals$ and $\alpha_i < \beta_i$ both real numbers. 

The corresponding non-linear edge agreement protocol can thus be expressed as 
{
\bea\label{nonlinear_edge_agreement}
\Sigma_{\ms \mc{F}}(\mc{G},\Phi): &&\left\{\begin{array}{ccc} \dot{x}_{\ms \mc{F}}(t) &=&- L_{\ms ess}(\mc{F})x_{\ms \mc{F}}(t) - L_e(\mc{F})R_{\ms (\mc{F},\mc{C})}P\left(\Phi(P^TR_{\ms (\mc{F},\mc{C})}^T x_{\ms \mc{F}}(t)) \right) + \Eforrest^T {\mathrm v}(t)  \\
z(t) &=&E(\mc{G}_{o})^T (\Eforrest^L)^T x_{\ms \mc{F}}(t) \end{array} \right. . \nonumber 
\eea
}

For both models $\Sigma_{\ms \mc{F}}(\mc{G},\Delta)$ and $\Sigma_{\ms \mc{F}}(\mc{G},\Phi)$ we will be concerned with determining bounds on the uncertainty that guarantee the \emph{robust stability} of the uncertain agreement protocols.  The main analytic tool will be the application of the small gain theorem.  To proceed with this analysis we first examine stability properties of the linear consensus protocol with arbitrary negative edge weights.  These results will then be applied to derive the more general robust stability statements of the uncertain models.

\section{On the Stability of Weighted Consensus}\label{sec:negweights_consensus}

Before examining the robust stability of the uncertain consensus networks presented in Section \ref{subsec:uncertain_edge}, we first study the general stability properties of the linear consensus protocol $\Sigma(\mc{G})$ with negative weights.  This section aims to reveal how both the magnitude and placement of negative weight edges in a weighted graph can influence the stability properties of $\Sigma(\mc{G})$.  In this section we assume there are no exogenous disturbances entering the protocol (i.e., $\w(t) = 0$).

The results of Section \ref{sec:edgelap} revealed that the signature of the weighted Laplacian is related to the signature of an associated matrix, $\Rf W \Rf^T$.  We now exploit the structure of this matrix to show how negative edge weights affect the stability of the weighted consensus.  We first examine the effect of negative edge weights for a class of graphs known as \emph{signed graphs}.  Then we proceed to provide a general stability characterization that turns out to be related to the notion of the \emph{effective resistance} of a graph.  Moreover, as $\Sigma(\mc{G})$ is a linear system, its stability can be determined by examining the eigenvalues of the weighted Laplacian $L(\mc{G})$.  In particular, the stability of $\Sigma(\mc{G})$ is guaranteed as long as the weighted Laplacian is a positive semi-definite matrix.  Therefore, the results to follow are presented in terms of the definiteness of $L(\mc{G})$.

\subsection{Signed Graphs and the Stability of $\Sigma(\mc{G})$}\label{subset:signedgraphs}

A useful approach for the analysis of graphs with both positive and negative edge weights is the notion of \emph{signed graphs} \cite{Zaslavsky1982}.  A signed graph is a weighted graph $\mc{G}$ with a partition of the edge set into edges with positive weights, denoted $\mc{E}_+$, and edges with negative weights, denoted $\mc{E}_-$. 
We can now define two subgraphs, one containing only positive weight edges, and one containing only negative weight edges; $\mc{G}_{+} = (\mc{V}, \mc{E}_{+}, \mc{W}_+)$ and $\mc{G}_{-} = (\mc{V}, \mc{E}_{-}, \mc{W}_{-})$.  The weight maps $\mc{W}_{+}$ and $\mc{W}_{-}$ are simply the original weight map $\mc{W}$ restricted to either $\mc{E}_+$ or $\mc{E}_{-}$ with corresponding diagonal weight matrices $W_+$ and $W_-$.  The corresponding incidence matrices can also be written as \cite{Zelazo2009b}
\beas
E_+ &\hspace{-7pt}:=&\hspace{-7pt} E(\mc{G}_+) = \Eforrestp\Rplus = \Eforrestp \leftm{cc} I & \Tplus \rightm \\
E_- &\hspace{-7pt}:=&\hspace{-7pt} E(\mc{G}_-) = \Eforrestm\Rminus = \Eforrestm \leftm{cc} I & \Tminus \rightm .
\eeas
A signed graph is \emph{balanced} if and only if the vertex set can be partitioned into 2 sets $\mc{V}_1$ and $\mc{V}_2$, such that all edges in $\mc{E}_-$ connect nodes in $\mc{V}_1$ to nodes in $\mc{V}_2$ (i.e., $\mc{G}_{-}$ is bipartite), and all edges in $\mc{E}_+$ connects nodes in $\mc{V}_i$ only to nodes in $\mc{V}_i$ for $i=1,2$ \cite{Harary1953}.%

Observe that the weighted graph Laplacian can now be expressed in terms of the positive and negative weight subgraphs,
\bea\label{signed_laplacian}
L(\mc{G}) = E_+W_+E_+^T-E_-|W_-|E_-^T.
\eea
This decomposition leads to the following statement on the definiteness of the weighted Laplacian.
\begin{theorem}\label{thm:psd2}
The weighted Laplacian is positive semi-definite if and only if 
\bea\label{lmi1}
\leftm{cc} |W_{-}|^{-1} & E_{-}^T  \\ E_{-} & E_+W_+E_+^T \rightm & \geq & 0.
\eea
\end{theorem}
\begin{proof}
This result follows from the Schur complement applied to (\ref{signed_laplacian}). \hfill \QEDclosed
\end{proof}

The linear matrix inequality (LMI) in (\ref{lmi1}) already indicates that the magnitude of the negative edge weights plays an important role in the definiteness of the Laplacian.  We now proceed to examine this LMI in more detail to reveal how also their location in the graph is important.
\begin{corollary}\label{cor:psd3}
The weighted Laplacian is positive semi-definite if and only if 
{
\bea \label{lmi_big}
\hspace{-12pt} \leftm{ccc}\hspace{-3pt} |W_{-}|^{-1} &  E_{-}^T(\Eforrestp^L)^T & E_{-}^TN_{\ms \mc{F}_+}  \hspace{-3pt}\\\hspace{-3pt}   \Eforrestp^L E_{-} &  \Rplus W_+ \Rplus^T  & 0 \hspace{-3pt}\\\hspace{-3pt} N_{\ms \mc{F}_+} ^TE_{-} & 0 & {\bf 0}\rightm &\hspace{-7pt}  \geq & \hspace{-7pt} 0.
\eea
}
\end{corollary}
\begin{proof}
Let $\mbox{\bf IM} [N_{\ms \mc{F_+}} ]= \mbox{\bf span}\{\mc{N}[\Eforrestp^T]\}$.
Consider the congruent transformation matrix 
{
\beas
Q &=& \leftm{cc} I & 0 \\ 0 & \leftm{cc} (\Eforrestp^L)^T &N_{\ms \mc{F}_+} \rightm \rightm. 
\eeas
}
Then the matrix in (\ref{lmi1}) has the same signature (by the congruent transformation using $S$) as the matrix in (\ref{lmi_big}).  From Theorem \ref{thm:psd2} it now follows that the weighted Laplacian is positive semi-definite if and only if (\ref{lmi_big}) is positive semi-definite.\hfill \QEDclosed
\end{proof}

We can make some more observations about the above result.  It directly follows that if $\mc{G}_+$ is connected, then $L(\mc{G}) \geq 0$ if and only if 
{
\bea\label{cor:psd4_lmi}
 \leftm{cc} |W_{-}|^{-1} &  E_{-}^T(\Eforrestp^L)^T  \\   \Eforrestp^L E_{-} &  \Rplus W_+ \Rplus^T   \rightm & \geq & 0.
\eea
}

We are now prepared to make a statement connecting the location of the negative edge weights to the definiteness of the weighted Laplacian.  First, we comment on the structural meaning of the matrix $E_{-}^TN_{\ms \mc{F}_+} $ in (\ref{lmi_big}). The quantity $E_{-}^TN_{\ms \mc{F}_+}$ characterizes any cuts of the original graph $\mc{G}$ using only the negative weighted edges.  To show this, assume $\mc{G}_+$ has $c$ connected components and let $N_{\ms \mc{F}_+} = \leftm{ccc} {\bf n}_1 & \cdots & {\bf n}_c\rightm$.  Then ${\bf n}_i$ is a $\{0,1\}$-vector with $[{\bf n}_i]_j = 1$ if and only if node $j \in \mc{V}$ is in component $i$.  Then it is clear that $E_{-}^T{\bf n}_i$ will be a $\{0,\pm 1\}$-vector with $[E_{-}^T{\bf n}_i]_k = \pm 1$ if and only if $k = (u,v) \in \mc{E}_{-}$ such that $u$ and $v$ are not in the same components of $\mc{G}_+$.

\begin{theorem}\label{thm:negativecuts}
Assume that $\mc{G}$ is connected and $\mc{E}_+,\mc{E}_{-} \neq \emptyset$.  If $\mc{G}_+$ is not connected, then the weighted Laplacian matrix is indefinite for any (non-zero) choice of negative edge weights.
\end{theorem}
\begin{proof}
Consider the following sub-matrix of (\ref{lmi_big}) obtained by deleting the center block row and column,
{
$$\leftm{cc}  |W_{-}|^{-1} &E_{-}^TN_{\ms \mc{F}_+}  \\ N_{\ms \mc{F}_+} ^TE_{-} & {\bf 0}  \rightm \in \reals^{(m+c) \times (m+c)} ,$$
}
where $m = |\mc{E}_{-}|$ and $c$ is the number of connected components of $\mc{G}_+$.
We assume that $\mc{G}_+$ is not connected, and thus $N_{\ms \mc{F}_+}$ contains $c$ columns.  Partition the matrix as $N_{\ms \mc{F}_+} = \leftm{ccc} {\bf n}_1 & \cdots & {\bf n}_c\rightm$ and recall that $[E_{-}^T{\bf n}_i]_k = \pm 1$ if and only if $k = (u,v) \in \mc{E}_{-}$ such that $u$ and $v$ are not in the same components of $\mc{G}_+$.  Denote by $\mbox{CUT}_i \subseteq \mc{E}_{-}$ as the set of negative weight edges used to form a cut with the $i$th component of $\mc{G}_+$, and let $\mbox{CUT} = \cup_i \mbox{CUT}_i$.  Then an expression for the quadratic form of the matrix of interest is 
{
\beas
x^T \leftm{cc}  |W_{-}|^{-1} &E_{-}^TN_{\ms \mc{F}_+}  \\ N_{\ms \mc{F}_+} ^TE_{-} & {\bf 0}  \rightm x &=&\hspace{-5pt} \sum_{i \in \mc{E}_{-} } |\mc{W}_{-}(i)|^{-1} x_i^2 +  \sum_{k \in \mbox{\small CUT}_1}  \pm 2 x_kx_{m+1} + \cdots + \sum_{k \in  \mbox{\small CUT}_c} \pm 2 x_kx_{m+c}.
\eeas
}
From the quadratic form, it is now clear that the elements of the vector $x_i$ for $i = m+1, \ldots, m+c$ can be arbitrarily chosen to make the inequality negative.  Therefore, there exists at least one negative eigenvalue and the matrix in (\ref{lmi_big}) is indefinite.  From Corollary \ref{cor:psd3} we can conclude that the weighted graph Laplacian is indefinite independent of the value of the negative weights. \hfill \QEDclosed
\end{proof}

Theorem \ref{thm:negativecuts} shows that if any of the negative weight edges forms a cut in the graph, then the Laplacian matrix must have negative eigenvalues.  A particular class of graphs satisfying the conditions of Theorem \ref{thm:negativecuts} are the balanced signed graphs.
\begin{corollary}\label{cor:balanced}
If a signed graph $\mc{G}$ is balanced then $L(\mc{G})$ is indefinite for any choice of negative edge weights.
\end{corollary}

\subsection{Effective Resistance and the Stability of $\Sigma(\mc{G})$}
The main result of Section \ref{subset:signedgraphs} provides an analytical justification of what may be considered an intuitive result.  That is, if the negative weight edges form a cut in the graph, then the weighted Laplacian will be indefinite; i.e., $\Sigma(\mc{G})$ will be unstable.  In this section, we reveal a more general condition on the negative edge weights that can lead to an indefinite weighted Laplacian.  This condition turns out to be related to the notion of the \emph{effective resistance} of a graph.  Results from this section were recently reported in \cite{Zelazo2014}, and thus the reader is referred to that work for related proofs.

It is well known that the weighted Laplacian of a graph can be interpreted as a resistor network \cite{Klein1993}.   
Each edge in the network can be thought of as a resistor with resistance equal to the inverse of the edge weight, $r_k = \mc{W}(k)^{-1}=w_k^{-1}$ for $k \in \mc{E}$.\footnote{Thus, the edge weight $w_k$ can be interpreted as an \emph{admittance}.}  The resistance between any two pairs of nodes can be determined using standard methods from electrical network theory \cite{Klein1993}.  It may also be computed using the Moore-Penrose pseudo-inverse of the graph Laplacian, denoted $L(\mc{G})^{\dagger}$.
\begin{definition}[\cite{Klein1993}]\label{def:effec_resistance}
The effective resistance between nodes $u,v \in \mc{V}$ in a weighted graph $\mc{G} = (\mc{V},\mc{E},\mc{W})$ is 
\beas
\mc{R}_{uv}(\mc{G}) &=& ({\bf e}_u - {\bf e}_v)^TL^{\dagger}(\mc{G})({\bf e}_u - {\bf e}_v) \\
&=& [L^{\dagger}(\mc{G})]_{uu}-2[L^{\dagger}(\mc{G})]_{uv}+[L^{\dagger}(\mc{G})]_{vv},
\eeas
where ${\bf e}_u$ is the indicator vector for node $u$, that is ${\bf e}_u = 1$ in the $u$ position and 0 elsewhere.
\end{definition}

Our first result shows how the effective resistance between two nodes is related to the matrix $\R W \R^T$ and the essential edge Laplacian. 
\begin{proposition}[\cite{Zelazo2014}]\label{prop:lap_pseudoinv}
Let $\mc{G}$ be a connected graph and assume $s(L(\mc{G}))=(n_+,n_-,1)$.  Then \bea\label{lap_pseudoinv}
L^{\dagger}(\mc{G})&=& (\Etree^L)^T\left(\R W\R^T\right)^{-1}\Etree^L = (\Etree^L)^{T} L_{\ms ess}(\mc{T})^{-1}\Etree^T.
\eea
\end{proposition}
\begin{proof}
From Theorem \ref{thm:signature1} we conclude the essential edge Laplacian is invertible and it follows that 
$$ L_{\ms ess}(\mc{T})^{-1} = \left(\R W\R^T\right)^{-1}L_e(\mc{T})^{-1},$$
and
$$L^{\dagger}(\mc{G}) = V \leftm{cc} \left(\R W\R^T\right)^{-1}L_e(\mc{T})^{-1} & {\bf 0} \\ {\bf 0} & 0 \rightm V^{-1},$$
where $S$ is the transformation matrix defined in Proposition \ref{prop:edgelap_sim2}, and (\ref{lap_pseudoinv}) follows directly. \QEDclosed
\end{proof}

From Proposition \ref{prop:lap_pseudoinv}, it is clear that the effective resistance between nodes $u,v \in \mc{V}$ can be expressed as
{
\bea\label{effec_res_rwr}
 \mc{R}_{uv}(\mc{G}) \hspace{-10pt}&=\hspace{-10pt}& ({\bf e}_u\hspace{-2pt} -\hspace{-2pt} {\bf e}_v)^T(\Etree^L)^T \left(\hspace{-2pt}\R W\R^T\hspace{-2pt}\right)^{-1}\hspace{-5pt}\Etree^L({\bf e}_u \hspace{-2pt}-\hspace{-2pt} {\bf e}_v). 
 \eea
}
We now present a result showing that this equivalent characterization of the effective resistance is useful for understanding the definiteness of the weighted Laplacian.

\begin{figure}[!t]
\begin{center}
\includegraphics[width=0.5\textwidth]{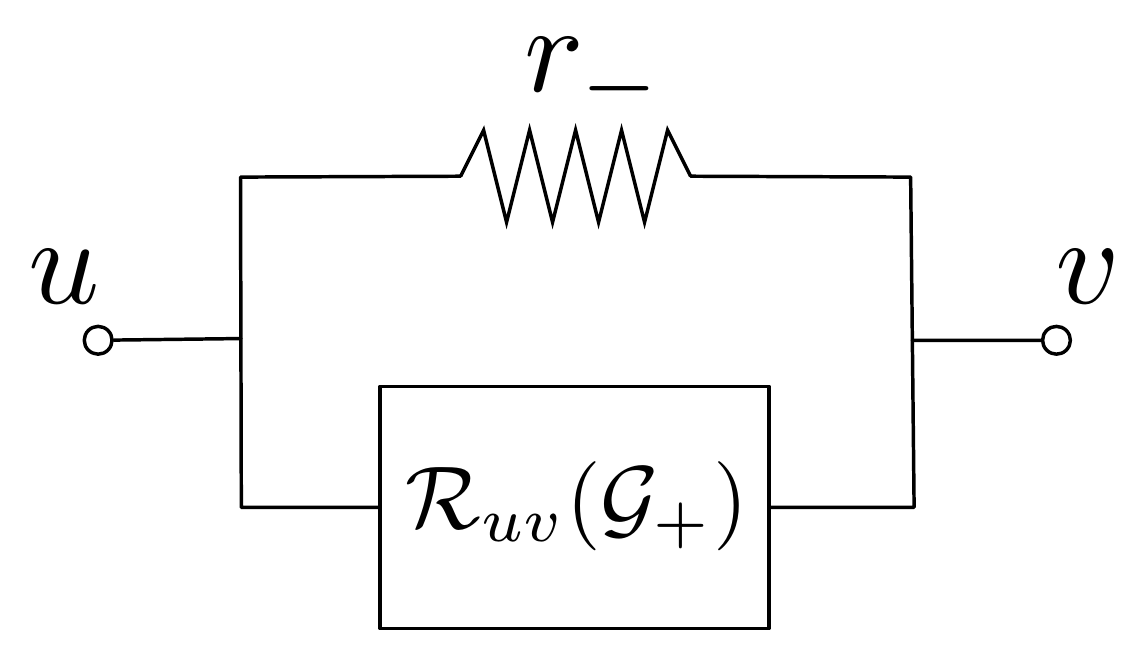}
\caption{Resistive network interpretation with one negative weight edge.}
\label{fig.effectiveresistance}
\end{center}
\end{figure}

\begin{theorem}[\cite{Zelazo2014}]\label{thm:one_negedge}
Assume that $\mc{G}_+$ is connected and $|\mc{E}_{-}|=1$ with $\mc{E}_{-} = \{e_{-}=(u,v)\}$.  Let $\mc{R}_{uv}(\mc{G}_+)$ denote the effective resistance between nodes $u,v \in \mc{V}$ over the graph $\mc{G}_+$.  Then $L(\mc{G})$ is positive semi-definite if and only if $|\mc{W}(e_{-})| \leq \mc{R}_{uv}^{-1}(\mc{G}_+)$.
\end{theorem}
\begin{proof}
From (\ref{signed_laplacian}) we have
$$L(\mc{G})= \Etreep\Rtplus W_+ \Rtplus^T\Etreep^T - E_-|\mc{W}(e_-)|E_-^T.$$
By the Schur complement, $L(\mc{G}) \geq 0$ if and only if
$$ \leftm{cc} |\mc{W}(e_-)|^{-1} & E_-^T \\ E_- & \Etreep\Rtplus W_+ \Rtplus^T\Etreep^T \rightm \geq 0.$$
Applying a congruent transformation to the above matrix using 
\beas
F &=& \leftm{cc} I & 0 \\ 0 & \leftm{cc} (\Etreep^L)^T &\ones \rightm \rightm
\eeas
leads to the following LMI condition,
$$ \leftm{cc} |\mc{W}(e_{-})|^{-1} &  E_{\_}^T(\Etreep^L)^T  \\   \Etreep^L E_{-} &  \Rtplus W_+ \Rtplus^T   \rightm  \geq 0.$$
Applying again the Schur complement, we obtain the equivalent condition that the matrix
{
$$ |\mc{W}(e_{-})|^{-1} - E_{-}^T(\Etreep^L)^T(\Rtplus W_+ \Rtplus^T)^{-1}\Etreep^L E_{\_}$$
}
must also be positive semi-definite.
Observe now that 
{
$$E_{-}^T(\Etreep^L)^T(\Rtplus W_+ \Rtplus^T)^{-1}\Etreep^L E_{-} = \mc{R}_{uv}(\mc{G}_+).$$
}
This then leads to the desired conclusion that $ |\mc{W}(e_{-})|  \leq  \mc{R}_{uv}^{-1}(\mc{G}_+)$.\hfill \QEDclosed
\end{proof}

The above result has a very intuitive physical interpretation.  The entire network $\mc{G}_+$ can be considered as a single lumped resistor between nodes $u$ and $v$ with resistance $\mc{R}_{uv}(\mc{G}_+)$.  The negative-weight edge can thus be thought of as adding another resistor in parallel between the nodes, as in Figure \ref{fig.effectiveresistance}.  The equivalent resistance between $u$ and $v$ is well-known to be 
$$\mc{R}_{uv}(\mc{G}) = \frac{\mc{R}_{uv}(\mc{G}_+)r_{-}}{\mc{R}_{uv}(\mc{G}_+)+r_{-}}.$$
If $r_{-}$ is a negative resistor, then choosing $r_{-} = -\mc{R}_{uv}(\mc{G}_+)$ corresponds to an equivalent resistance that is infinite, i.e., an \emph{open circuit}.  The open circuit can be thought of as a cut between the terminals $u$ and $v$.  This interpretation is inline also with the results of Theorem \ref{thm:negativecuts}.  Indeed, in this case the graph $\mc{G}_+$ is not connected and the effective resistance between any two nodes connected by a negative edge weight over $\mc{G}_+$ is infinite, i.e., $\mc{R}_{uv}^{-1}(\mc{G}_+)=0$.

The result in Theorem \ref{thm:one_negedge} can be generalized to multiple negative weight edges with some additional assumptions on how those edges are distributed in the graph.  In this direction, for each edge $k=(u,v)\in \mc{E}_-$, define the set $\mc{P}_{k} \subseteq \mc{E}_+$ to be the set of all edges in $\mc{G}_+$ that belong to a path connecting nodes $u$ to $v$, 
Let $\mc{G}_+(\mc{P}_{k}) \subseteq \mc{G}_+$ be the subgraph induced by the edges in $\mc{P}_{k}$.\footnote{Thus, $\mc{G}_+(\mc{P}_{k}) = (\mc{V}(\mc{P}_{k}),\mc{P}_{k})$ where $\mc{V}(\mc{P}_{k})\subseteq \mc{V}$ are the nodes incident to edges in $P_{k}$.}
Note that if $\mc{P}_{k} \cap \mc{P}_{{k'}} = \emptyset$ for edges with distinct nodes (i.e., $k=(u,v)$ and ${k'}=({u'},{v'}) \in \mc{E}_-$), then there exists no cycle in $\mc{G}_+$ containing the nodes $u,v,{u'},{v'}$.
An important class of graphs that can admit such a partition are the \emph{cactus graphs} \cite{Markov2007}.  Using this characterization, the following statement on effective resistance with multiple negative weight edges follows.

\begin{theorem}[\cite{Zelazo2014}]\label{thm:multiple_edge}
Assume that $\mc{G}_+$ is connected and $|\mc{E}_-|>1$.  Let $\mc{R}_{k}(\mc{G}_+)$ denote the effective resistance between nodes $u,v \in \mc{V}$ over the graph $\mc{G}_+$ with $k=(u,v)\in \mc{E}_-$, and let $\mc{\bf R} = \mbox{\bf diag}\{\mc{R}_1(\mc{G}_+),\ldots,\mc{R}_{\ms |\mc{E}_-|}(\mc{G}_+)\}$.  Furthermore, assume that $\mc{P}_i \cap \mc{P}_j = \emptyset$ for all $i,j \in \mc{E}_-$, where $\mc{P}_i$ is the set of all edges belonging to a path connecting the nodes incident to edge $i \in \mc{E}_{\ms \mc{\Delta}}$. 
Then the weighted Laplacian is positive semi-definite if and only if $|W_-| \leq \mc{\bf R}^{-1}$.\end{theorem}
\begin{proof}
As in the proof of Theorem \ref{thm:one_negedge}, we consider the LMI
{
\bea\label{schurlmi}
|W_-|^{-1} - E_{-}^T(\Etreep^L)^T(\Rtplus W_+ \Rtplus^T)^{-1}\Etreep^L E_{-} \geq 0 
\eea
}
Due to the location of the negative weight edges assumed in the graph, it can be verified that the matrix $E_{-}^T(\Etreep^L)^T(\Rtplus W_+ \Rtplus^T)^{-1}\Etreep^L E_{-}$ is in fact a diagonal matrix with $\mc{R}_{k}(\mc{G}_+)$ for $k=1,\ldots, |\mc{E}_-|$ on the diagonal, denoted as ${\bf R}$.  To see this, observe that $\Etreep^LE_{-}$ is a $\{0,\pm 1\}$ matrix that describes which edges in the spanning tree $\mc{T}_+$ can create a cycle with each edge in $\mc{E}_{-}$ (this is related to the matrix $\Tcycle$ used in Proposition \ref{prop:edgelap_sim2} since $\mc{T}_+=\mc{T}$ and therefore $\mc{E}_- \subseteq \mc{E}_c$).  Observe also that an edge $k \in \mc{E}_-$ can only be incident to nodes in the subgraph $\mc{G}_+(P_k)$.  Therefore, the matrix $\Etreep^LE_{-}$ has a partitioned structure (after a suitable relabeling of the edges) such that the $k$th column of $\Etreep^LE_{-}$ will contain non-zero elements corresponding to edges in $\mc{T}_+ \cap \mc{G}_+(P_k)$.
The LMI condition can now be expressed as $|W_-|^{-1} \geq {\bf R}$ which implies that $|W_-| \leq  {\bf R}^{-1}$ concluding the proof.\QEDclosed
\end{proof}

Theorem \ref{thm:multiple_edge} also has the same physical interpretation as Theorem \ref{thm:one_negedge}.  Indeed, the resistance between two nodes contained in a sub-graph $\mc{G}_+(\mc{P}_k)$ is not determined by any other edges in the network.  Both Theorems \ref{thm:one_negedge} and \ref{thm:multiple_edge} provide a clear characterization of how negative weight edges can impact the definiteness of the weighted Laplacian, and how that is related to the effective resistance in the graph.  In fact, from (\ref{cor:psd4_lmi}) we also can observe an additional fact relating the total effective resistance between all nodes incident to edges in $\mc{E}_-$ and the definiteness of the graph, independent of the actual location of these edges in the network.

\begin{corollary}[\cite{Zelazo2014}]\label{cor:total_resist}
Assume that $\mc{G}_+$ is connected.  If $L(\mc{G})\geq 0$, then  
$ \sum_{k \in \mc{E}_-} |\mc{W}(k)|^{-1} \geq {\bf R}_{tot}^{\ms \mc{E}_-},$
where 
{
$${\bf R}_{tot}^{\ms \mc{E}_-}\hspace{-3pt}=\hspace{-3pt}\mbox{\bf trace}\hspace{-3pt}\left[ E_{-}^T(\Etreep^L)^T(\Rtplus W_+ \Rtplus^T)^{-1}\Etreep^L E_{-}\right] .$$
}
\end{corollary}
\begin{proof}
This follows directly from the trace of (\ref{schurlmi}).
\end{proof}

Corollary \ref{cor:total_resist} indicates that a weighted Laplacian with negative weights can still be positive semi-definite, and in that case the total magnitude of the negative weight edges is closely related to the total effective resistance in the network (defined over the nodes incident to $\mc{E}_-$). The notion of total effective resistance has also appeared in works characterizing the $\mc{H}_2$ performance of certain multi-agent networks \cite{Bamieh2009, Barooah2006a, Siami2013}.  While Corollary \ref{cor:total_resist} only provides a sufficient condition for the definiteness of the weighted Laplacian, it nevertheless reinforces its connection to the notion of effective resistance.

\section{On the Robust Stability of the Uncertain Consensus Protocol}\label{sec:small gain}

The results of the previous section provide the correct foundation to now consider the robust stability of the uncertain consensus models presented in Section \ref{subsec:uncertain_edge}.  A natural approach for analyzing the robust stability is via the celebrated \emph{small gain theorem}.  While the small gain theorem often provides very conservative results, we demonstrate in this section that for certain classes of uncertainties the small gain result is in fact an exact condition.

\subsection{Robust Stability of $\Sigma_{\ms \mc{F}}(\mc{G},\Delta)$}
We now proceed with an analysis of the system $\Sigma_{\ms \mc{F}}(\mc{G},\Delta)$ in (\ref{uncertain_edge_agreement}).  Based on the system interconnection shown in Figure \ref{fig:uncertain}, the map from the exogenous inputs $\w(t)$ to the controlled output $z(t)$ in the presence of the structured uncertainty $\Delta \in {\bf \Delta}$ can be characterized by the upper fractional transformation \cite{Dullerud2000},
\bea \label{UFT}
\overline{S}(\Sigma_{\ms \mc{F}}(\mc{G}),\Delta) = M_{22}+M_{21}\Delta \left(I-M_{11}\Delta \right)^{-1}M_{12},
\eea
where
{
\beas
M_{11}(s) &\hspace{-5pt}=&\hspace{-5pt} P^TR_{\ms (\mc{F},\mc{C})}^T(sI+L_{\ms ess}(\mc{F}) )^{-1}L_e(\mc{F})R_{\ms (\mc{F},\mc{C})}P \\
M_{12}(s) &\hspace{-5pt}=& \hspace{-5pt}P^TR_{\ms (\mc{F},\mc{C})}^T(sI+L_{\ms ess}(\mc{F}) )^{-1}E(\mc{F})^T \\
M_{21}(s) &\hspace{-5pt}=&\hspace{-5pt} E(\mc{G}_{o})^T (\Eforrest^L)^T(sI+L_{\ms ess}(\mc{F}) )^{-1}L_e(\mc{F})R_{\ms (\mc{F},\mc{C})}P  \\
M_{22}(s) &\hspace{-5pt}=&\hspace{-5pt} E(\mc{G}_{o})^T (\Eforrest^L)^T(sI+L_{\ms ess}(\mc{F}) )^{-1}E(\mc{F})^T.
\eeas
}
\indent This representation can lead directly to a \emph{small-gain} interpretation for the allowable edge-weight uncertainties that guarantees the system is robustly stable.  In particular, a sufficient condition for determining whether $(I-M_{11}\Delta)$ has a stable proper inverse is to ensure that 
$\|M_{11}(s)\Delta\|_{\infty} < 1.$

We now cite a result from \cite{Zelazo2009b} that gives insight on the $\mc{H}_{\infty}$ norm of the transfer function matrix $M_{11}(s)$.

\begin{proposition}[\cite{Zelazo2009b}]\label{prop:tf_norm}
%
$$\|M_{11}(s) \|_{\infty} = \overline{\sigma}(M_{11}(0)).$$
\end{proposition}
This results shows that the $\mc{H}_{\infty}$ performance of the system $M_{11}(s)$ can be obtained by computing the largest singular value of the real matrix $M_{11}(0)$, leading to the result on the robust stability of $\Sigma_{\ms \mc{F}}(\mc{G},\Delta)$.  

\begin{theorem}\label{thm:general_uncertain}
Let ${\bf \Delta} = \{ \Delta \in \reals^{|\mc{E}_{\ms \Delta}| \times |\mc{E}_{\ms \Delta}|} , \; \|\Delta\|\leq \overline \delta, \, \Delta \mbox{ diagonal}\}$ and assume $\Sigma_{\ms \mc{F}}(\mc{G})$ is nominally stable.  Then the uncertain edge agreement protocol is robustly stable for any $\Delta \in {\bf \Delta}$ if 
$\|\Delta\|  < ({\sigma}(M_{11}(0)))^{-1}.$
\end{theorem}

Theorem \ref{thm:general_uncertain} is in fact a direct statement of the small-gain theorem and can be considered conservative.  One might wonder if the results of Theorem \ref{thm:negativecuts} might lead to a less conservative result.  That is, if $\mc{E}_{\ms \Delta}$ forms a cut set of $\mc{G}$, is it possible that $\max_{e \in \mc{E}_{\ms \Delta}} \mc{W}(e) < \overline{\sigma}(M_{11}(0))^{-1}$.  The following result shows that this can not be the case.

\begin{proposition}\label{prop:M11bound}
$$ \left(\max_{e \in \mc{E}_{\ms \Delta}} \mc{W}(e)\right)^{-1} \leq \max_{e \in \mc{E}_{\ms \Delta}} \mc{R}_{e}(\mc{G}) \leq \overline{\sigma}(M_{11}(0)) \leq  {\bf R}_{tot}^{\mc{E}_{\ms \Delta}},$$
where 
{
$${\bf R}_{tot}^{\mc{E}_{\ms \Delta}} = \mbox{\bf trace}\left[ P^TE^T(\Etree^L)^T(\R W \R^T)^{-1}\Etree^L EP\right].$$}
\end{proposition}
\begin{proof}
Let $E_{\ms \Delta}=EP$ denote the incidence matrix with columns corresponding to the edges in $\mc{E}_{\ms \Delta}$.  Observe that $R_{\ms (\mc{F},\mc{C})}P = \Etree^LE_{\ms \Delta}$.  Therefore, the diagonal entries of $M_{11}(0)$ correspond to the effective resistance between nodes incident to the edges in $\mc{E}_{\ms \Delta}$ (from (\ref{effec_res_rwr})), establishing a lower bound on $\overline{\sigma}(M_{11}(0))$.  By a similar argument, it follows that $\mbox{\bf trace}[M_{11}(0)] = {\bf R}_{tot}^{\mc{E}_{\ms \Delta}}$ leading to the upper bound.  

To establish the lower bound, let $S, S' \subset \mc{V}$ be the set of nodes incident to edges in $\mc{E}_{\ms \Delta}$ (with $S$ the nodes that are the head of each edge, and $S'$ the tail).   For each edge $e \in \mc{E}_{\ms \Delta}$, let $r_e = \mc{W}(e)^{-1} = w_e$ denote its resistance.
Consider now the edges $e_1=(u,v), e_2=(u',v') \in \mc{E}_{\ms \Delta}$, and assume $r_1 \geq r_2$ (equivalently, $w_1 \leq w_2$).  Similarly, let $\tilde{\mc{R}}_{uv}=\mc{R}_{uv}(\mc{G}\setminus e_1)$ denote the effective resistance between $u$ and $v$ after edge $e_1$ is removed.  It now follows that
{
\beas
\mc{R}_{uv}(\mc{G}) &=& \frac{r_1\tilde{\mc{R}}_{uv}}{r_1+\tilde{\mc{R}}_{uv}} \geq \frac{r_2(r_1+\tilde{\mc{R}}_{uv}-r_1)}{r_1+\tilde{\mc{R}}_{uv}} \geq r_2,
\eeas
}
Since the above inequality holds for any pair of edges in $\mc{E}_{\ms \Delta}$, and in particular when $r_2 = (\max_{e\in \mc{E}_{\ms \Delta}} \mc{W}(e))^{-1}$, the lower bound is verified. \hfill \QEDclosed
\end{proof}

By introducing additional assumptions into the system $\Sigma_{\ms \mc{F}}(\mc{G},\Delta)$, the small gain result of Theorem \ref{thm:general_uncertain} can be tightened and endowed with more meaningful physical interpretations.

\begin{theorem}\label{thm:alluncertain}
Consider the uncertain edge agreement protocol $\Sigma_{\ms \mc{F}}(\mc{G},\Delta)$ with $\mc{E}_{\ms \Delta} = \mc{E}$ and the nominal edge weights are all positive and equal, i.e., $W = \alpha I$ with $\alpha > 0$.  Then $\|M_{11}(s)\|_{\infty} = {\alpha}^{-1}$
and $\Sigma_{\ms \mc{F}}(\mc{G},\Delta)$ is robustly stable for all diagonal $\Delta$ satisfying
$\|\Delta \|_{\infty} = \overline{\sigma}(\Delta) < \alpha.$

\end{theorem}
\begin{proof}
With $\mc{E}_{\ms \Delta} = \mc{E}$ it follows that $P = I$.  Therefore, 
$M_{11}(0) = \frac{1}{\alpha} R_{\ms (\mc{F},\mc{C})}^T(R_{\ms (\mc{F},\mc{C})}R_{\ms (\mc{F},\mc{C})}^T)^{-1}R_{\ms (\mc{F},\mc{C})}$
is a projection matrix scaled by $\alpha^{-1}$, that is $\|M_{11}(s)\|_{\infty}=\overline{\sigma}(M_{11}(0))=\alpha^{-1}$. The bound on the uncertainty $\Delta$ then follows immediately from the small-gain theorem. \hfill \QEDclosed
\end{proof}

While still a conservative result, Theorem \ref{thm:alluncertain} provides has a more intuitive interpretation.  It states that all edge weights must remain positive even in the presence of uncertainty.  This result can also be understood in the context of Theorem \ref{thm:negativecuts} since one can imagine the uncertainty attacking a set of edges that forms a cut in the graph.  Thus, if the uncertainty is such that the edges in the cut have negative weight (i.e., $w_k + \delta_k < 0$) then the resulting edge agreement protocol will be unstable.

The small gain result can in fact lead to an exact condition for the robust stability of $\Sigma_{\ms \mc{F}}(\mc{G},\Delta)$.  The next result considers the uncertain edge agreement problem when the uncertainty is present in only a single edge in the graph.

\begin{theorem}\label{thm:oneuncertain}
Consider the uncertain edge agreement protocol $\Sigma_{\ms \mc{F}}(\mc{G},\Delta)$ with $\mc{E}_{\ms \Delta} = \{\{u,v\}\}$ (i.e., $|\mc{E}_{\ms \Delta}| = 1$) and assume $\Sigma(\mc{G})$ is nominally stable.  
Then $\|M_{11}(s) \|_{\infty} = \mc{R}_{uv}(\mc{G})$ and $\Sigma(\mc{G},\Delta)$ is robustly stable for all $\delta$ satisfying
$\|\Delta \|_{\infty}  < \mc{R}_{uv}^{-1}(\mc{G}).$
\end{theorem}
\begin{proof}
This statement directly follows Theorem \ref{thm:one_negedge}, Proposition \ref{prop:tf_norm}, and the effective resistance between two nodes in a graph (\ref{lap_pseudoinv}).  \hfill \QEDclosed
\end{proof}

Theorem \ref{thm:oneuncertain} provides an exact condition for the robust stability of $\Sigma_{\ms \mc{F}}(\mc{G},\Delta)$ when a single edge has uncertainty.   That is, there is no conservatism in this result, which is a direct consequence of Theorem \ref{thm:one_negedge}.  This result also demonstrates that the effective resistance between nodes in the network can serve the role of a \emph{gain margin} for the system.

Consider now a scenario where an attacker is able to perturb the weight of any single edge in the graph.  From Theorem \ref{thm:oneuncertain}, each edge can tolerate a perturbation determined by the effective resistance between the two incident nodes.  Thus, an attacker might consider choosing the edge with the `smallest' margin.  This leads to a less restrictive uncertainty class than in Theorem \ref{thm:oneuncertain}, however, the result will consequently be more conservative.

\begin{corollary}\label{cor:smallest}
Let ${\bf \Delta} = \{ \delta P_iP_i^T, \, i=1,\ldots , |\mc{E}|, \, |\delta| \leq \overline{\delta} \}$ be the set of allowable perturbations on the edge weights, where $P_i \in \reals^{|\mc{E}|}$ satisfies $[P_i]_j=1$ when $j=i$ and 0 otherwise, and assume that $\Sigma_{\ms \mc{F}}(\mc{G})$ is nominally stable.  Then $\Sigma(\mc{G},\Delta)$ is robustly stable if for any $\Delta \in {\bf \Delta}$ one has
$\|\Delta \|_{\infty}   < \min_{uv} \mc{R}_{uv}^{-1}(\mc{G}).$
\end{corollary}

Corollary \ref{cor:smallest} demonstrates the earlier discussion.  If only a single edge can be ``attacked," then the robustness margin of the entire network will be determined by the edge connecting the nodes with the largest effective resistance between them.  It should be noted that this value will depend on both the nominal values of the edge weights along with the location of the edges in the graph.

Using the result of Theorem \ref{thm:multiple_edge} also leads to a less conservative corollary for Theorem \ref{thm:general_uncertain}.

\begin{corollary}
Let $\mc{E}_{\ms \Delta}$ denote the set of uncertain edges and define the uncertainty set as ${\bf \Delta} = \{ \Delta \in \reals^{|\mc{E}_{\ms \Delta}| \times |\mc{E}_{\ms \Delta}|} , \; \|\Delta\|\leq \overline \delta, \, \Delta \mbox{ diagonal}\}$.  For each edge $k = (u,v) \in \mc{E}_{\ms \Delta}$, define $\mc{P}_k \subseteq \mc{E}$ as the set of all edges in $\mc{G}$ that belong to a path connecting nodes $u$ to $v$ 
 and assume that $\mc{P}_i \cap \mc{P}_j = \emptyset$ for all $i,j \in \mc{E}_{\ms \Delta}$ with $\mc{\bf R} = \mbox{\bf diag}\{\mc{R}_1(\mc{G}),\ldots,\mc{R}_{\ms |\mc{E}_{\ms \Delta}|}(\mc{G})\}$.  Assume that $\Sigma_{\ms \mc{F}}(\mc{G})$ is nominally stable.  Then $\Sigma(\mc{G},\Delta)$ is robustly stable for all $\Delta \in {\bf \Delta}$ if
$ \|\Delta \|< \min_{e \in \mc{E}_{\ms \Delta}} \mc{R}_e(\mc{G})^{-1}.$ 
\end{corollary}

\subsection{Robust Stability of $\Sigma_{\ms \mc{F}}(\mc{G},\Phi)$}

We now examine the robust stability of the consensus system with nonlinear couplings $\Sigma_{\ms \mc{F}}(\mc{G},{\Phi})$ given in (\ref{nonlinear_edge_agreement}).  First, we recall results on the multivariable circle criterion \cite{Haddad1993}.  

Let $K_1, K_2 \in \reals^{|\mc{E}_{\ms \Delta}| \times |\mc{E}_{\ms \Delta}|}$ be real diagonal matrices with $K_2 - K_1 > 0$. Define ${\bf \Phi}$ to be the set of all sector-bounded time-varying nonlinearities as described in Section \ref{subsec:uncertain_edge} with $K_1 = \mbox{\bf diag}\{\alpha_1, \ldots, \alpha_{\ms |\mc{E}_{\ms \Delta}|} \}$ and $K_2 = \mbox{\bf diag}\{\beta_1, \ldots, \beta_{\ms |\mc{E}_{\ms \Delta}|} \}$.  In the following, we denote by $G(s)$ the transfer-function of a linear system with state-space matrices $(A,B,C,D)$.

\begin{lemma}[\cite{Haddad1993}]\label{lem:SPR}
The following statements are equivalent:
\begin{enumerate}
	\item[(i)] The state matrix $A$ is asymptotically stable and $G(s)$ is strictly positive real;
	\item[(ii)] $D+D^T > 0$ and there exists $X,Y>0$ such that 
	{
	$$0=A^TX+XA+(B^TX-C)^T(D+D^T)^{-1}(B^TX-C)+Y.$$
	}
\end{enumerate}
\end{lemma}
\begin{theorem}[\cite{Haddad1993}]\label{thm:sectorstable}
If $(I+K_2G(s))(I+K_1G(s))^{-1}$ is strictly positive real, then for all $\Phi \in {\bf \Phi}$ the negative feedback interconnection of $G(s)$ and $\Phi$ is asymptotically stable.
\end{theorem}

We are now prepared to make a general statement on the robust stability of $\Sigma_{\ms \mc{F}}(\mc{G},\Phi)$.
\begin{theorem}
Consider the nonlinear edge agreement protocol $\Sigma_{\ms \mc{F}}(\mc{G},\Phi)$ with $\mc{E}_{\ms \Delta} \subseteq \mc{E}$ and assume $\Sigma(\mc{G})$ is nominally stable.  Then $\Sigma_{\ms \mc{F}}(\mc{G},\Phi)$ is asymptotically stable for any $\Phi \in {\bf \Phi}$ if 
$$\|K_1\| = \max_i |\alpha_i| <\frac{1}{ \|M_{11}(s)\|_{\infty}},$$
and
{ $2W + P\left(K - (1-\sqrt{2})I  \right)\left(K -(1+\sqrt{2}) I  \right)P^T > 0,$}
where $K = K_2-K_1$.
\end{theorem}
\begin{proof}
From Lemma \ref{lem:SPR} it follows that a necessary condition for $(I+K_2M_{11}(s))(I+K_1M_{11}(s))^{-1}$ to be strictly positive real is for $(I+K_1M_{11}(s))^{-1}$ to have a stable and proper inverse.  Since $K_1$ may have negative entries, we arrive at the small-gain result on the largest singular value $K_1$ to guarantee the inverse exists.  
Next, it can be shown that a minimal state-space realization of the transfer function $(I+K_2M_{11}(s))(I+K_1M_{11}(s))^{-1}$ can be expressed by the state-space matrices 
\beas
\tilde{A} &\hspace{-8pt}=&\hspace{-8pt} -L_{\ms ess}(\mc{F})-L_e(\mc{F})\Rf P K_1 P^T\Rf^T \\
\tilde{B} &\hspace{-8pt}=&\hspace{-8pt} L_e(\mc{F})\Rf P, \, \tilde{C} = (K_2-K_1)P^T \Rf^T, \, \tilde{D} = I.
\eeas
Using Theorem \ref{thm:sectorstable}, it can be verified that choosing $X = L_e(\mc{F})^{-1}$ and $Y = \Rf Q \Rf^T$ with
$$Q = 2W-PP^T -P(K_2-K_1)^2P+2P(K_2-K_1)P^T$$
satisfies the equality stated in condition (ii).  Thus, $Y > 0$  if and only if $Q > 0$ concluding the proof.\hfill \QEDclosed
\end{proof}

As expected, the agreement protocol with sector non-linearities on the edges will inherit certain robustness measures found in the case where there are additive perturbations.  We can understand the above result better when considering the case where $|\mc{E}_{\ms \Delta}| = 1$; that is the nonlinearity is present on only a single edge.

\begin{corollary}\label{cor:nonlinear_singleedge}
Consider the nonlinear edge agreement protocol $\Sigma_{\ms \mc{F}}(\mc{G},\Phi)$ with $\mc{E}_{\ms \Delta} = \{\{u,v\}\}$ (i.e., $|\mc{E}_{\ms \Delta}| = 1$) and assume $\Sigma(\mc{G})$ is nominally stable.  Then $\Sigma_{\ms \mc{F}}(\mc{G},\Phi)$ is asymptotically stable for all $\Phi \in {\bf \Phi}$ satisfying
$$|\alpha| < \mc{R}_{uv}^{-1}(\mc{G}) \mbox{ and } ((\beta-\alpha)^2-2(\beta-\alpha)-1)> -2w_{uv}.$$ 
\end{corollary}

The result of Corollary \ref{cor:nonlinear_singleedge} can also be interpreted in terms of passivity indices quantifying the shortage or excess of passivity in the system.  Indeed, Corollary \ref{cor:nonlinear_singleedge} can be used to conclude that the effective resistance between a pair of nodes with non-linear couplings represents an \emph{excess of passivity} for the system.

\section{Simulation Example}\label{sec:sims}
We now briefly report on a numerical simulation illustrating the main results of this work.  Figure \ref{fig:75node} shows a random geometric graph with 75 nodes.  The edge weights are taken to be the inverse of the Euclidean distance between neighboring nodes.  With $\mc{E}_{\ms \Delta} = \mc{E}$, Corollary \ref{cor:smallest} states that the robustness margin corresponds to the inverse of the largest effective resistance between nodes in $\mc{E}_{\ms \Delta}$, which in this example corresponds to the edge connecting the two black nodes in Figure \ref{fig:75node}; the effective resistance between these nodes is computed to be 0.429.  Figure \ref{fig:cluster_negweight} shows the resulting trajectories when the uncertainty on that edge exactly equals the inverse of the effective resistance, leading to a \emph{clustering phenomena}.  This was examined in more detail in \cite{Zelazo2014}.

Using the same graph, Figures \ref{fig:sectors} shows two non-linear coupling functions used across the uncertain edge highlighted earlier ($y = ax + \sin{x}$).  Choosing the nonlinearity satisfying the sector condition from Corollary \ref{cor:nonlinear_singleedge} ($a<-2.035$) leads to stable trajectories (Figure \ref{fig:sector_stable}).  On the other hand, choosing $a=-3$ (red line in Figure \ref{fig:sectors}) leads to unstable dynamics.

\begin{figure*}[!t]
\begin{center}
	\subfigure[Random geometric graph.] {\scalebox{.4}{\includegraphics{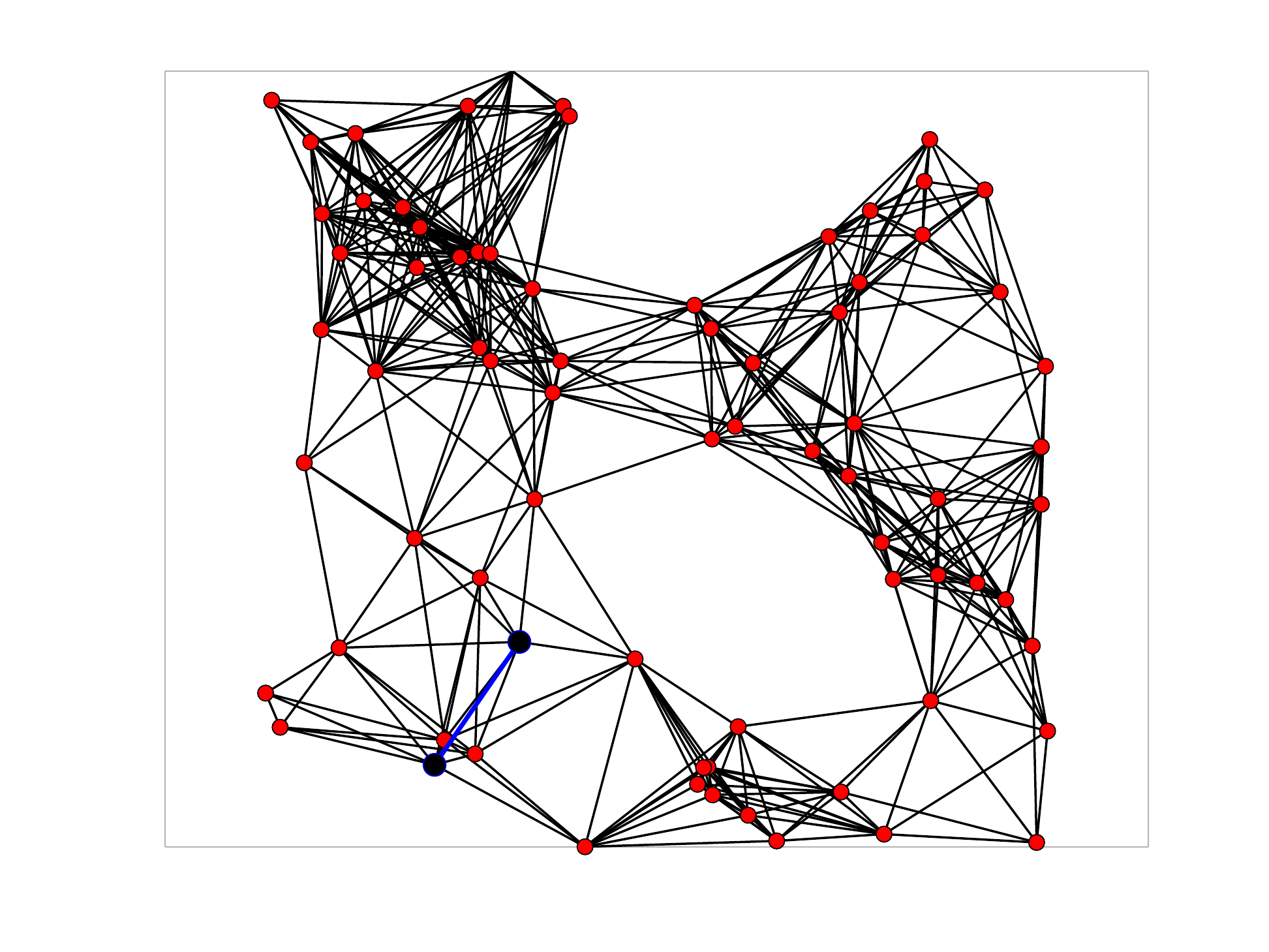}}\label{fig:75node}}
	\subfigure[Clustering results from an edge weight perturbation.]{\scalebox{.4}{\includegraphics{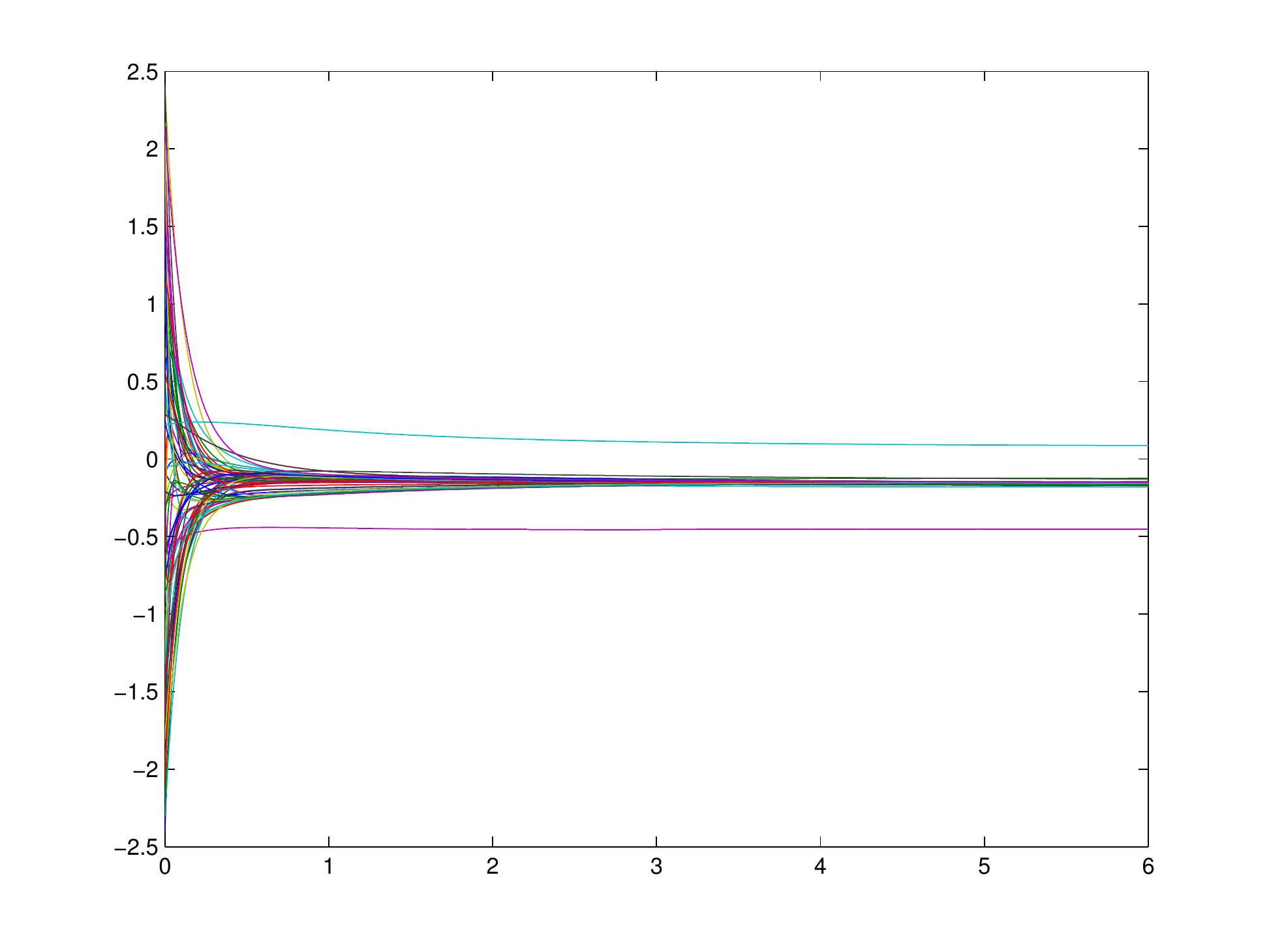}}\label{fig:cluster_negweight}}
	\subfigure[Nonlinear coupling functions.]{\scalebox{.4}{\includegraphics{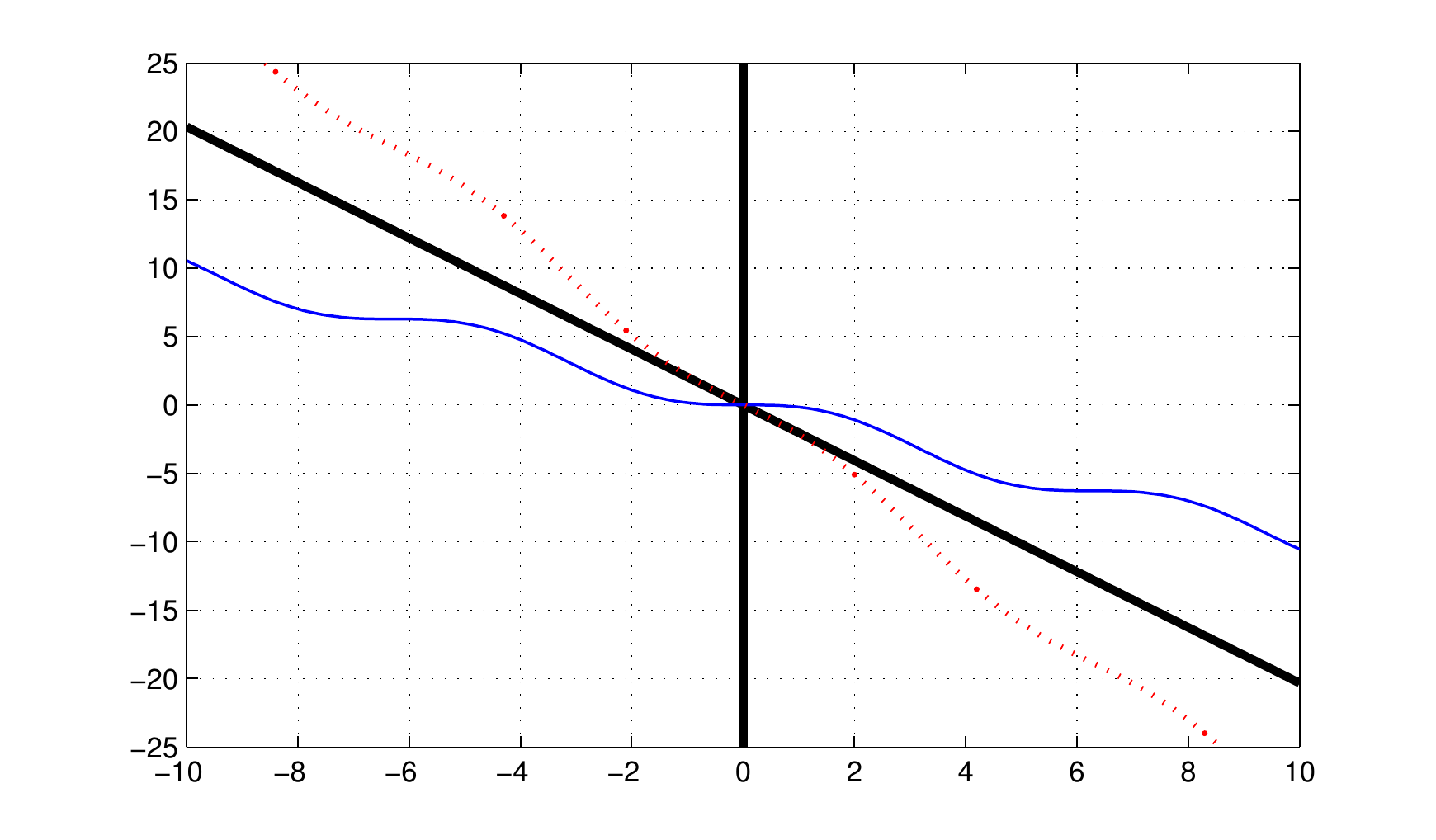}}\label{fig:sectors}}
	\subfigure[Nonlinearity satisfying Corollary \ref{cor:nonlinear_singleedge}.]{\scalebox{.4}{\includegraphics{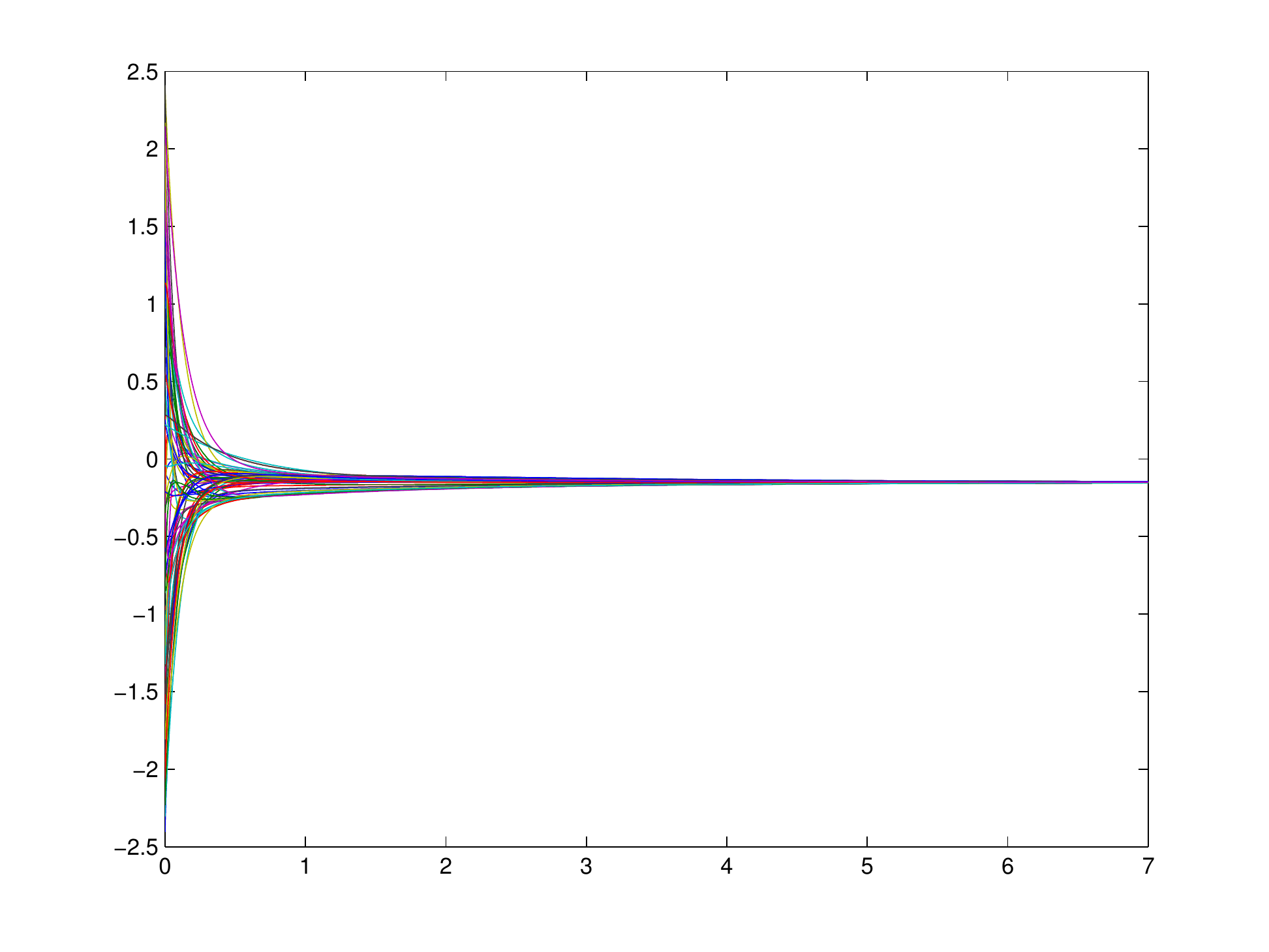}}\label{fig:sector_stable}}
  \caption{A network with an uncertain edge weight.}\label{fig:negweights_consensus}
\end{center}
\end{figure*}


\section{Concluding Remarks}\label{sec:conclusion}
This work examined the robust stability of weighted consensus protocols with bounded additive uncertainties on the edge weights.  The main results demonstrate that the robustness margins of such systems are determined by both the combinatorial properties of the uncertain edges (i.e., where in the network they are located), and the nominal magnitude of the edge weights.  These margins were related to the notion of the effective resistance in a network, and the robust stability results were cast in this framework.  We believe that this framework provides a new graph-theoretic interpretation for classical notions from robust control theory when applied to networked systems.  The ability for this framework to also handle non-linear extensions suggest a greater utility and our future works aim to apply these results to certain real-world applications.

{
\section*{Acknowledgements}
The work presented here has been supported by the Arlene \& Arnold Goldstein Center at the Technion Autonomous System Program (TASP) and the Israel Science Foundation.

{
 \bibliographystyle{IEEEtran}
\bibliography{LaplacianMatrix}
}
}
\end{document}